\documentclass[12pt]{amsart}

\usepackage{pinlabel}
\usepackage{amsmath}
\usepackage{amssymb}
\usepackage{amsthm}
\usepackage[all,matrix]{xy}
\usepackage{graphicx}
\usepackage[bf,up,centerlast]{caption}

\usepackage[retainorgcmds]{IEEEtrantools}

\newtheorem{Teo}{Theorem}[section]
\newtheorem{prop}[Teo]{Proposition}
\newtheorem{theo}[Teo]{Theorem}

\newtheorem{coro}[Teo]{Corollary}

\theoremstyle{definition}
\newtheorem{definition}[Teo]{Definition}
\newtheorem{ques}[Teo]{Question}
\newtheorem{claim}[Teo]{Claim}

\theoremstyle{remark}
\newtheorem{remark}[Teo]{Remark}

\newcommand{\bd}{\partial}

\newcommand{\F}{\mathcal{F}}
\newcommand{\Su}{\mathcal{S}}

\begin{document}

\keywords{}

\title{Additivity of circular width}

\author{M. Eudave-Mu\~noz}
\address{ \hskip-\parindent
Mario Eudave-Mu\~noz \\
 Instituto de Matem\'aticas\\
 Universidad Nacional Aut\'onoma de M\'exico\\
  M\'exico, D.F.\\
 MX}
\email{mario@matem.unam.mx}

\author{F. Manjarrez-Guti\'errez}
\address{ \hskip-\parindent
Fabiola Manjarrez-Guti\'errez \\
 Instituto de Matem\'aticas\\
 Universidad Nacional Aut\'onoma de M\'exico\\
  M\'exico, D.F.\\
 MX}
\email{fabiola@matem.unam.mx}

\thanks{Research supported by UNAM }
\date{\today}
\subjclass{57M25}

\begin{abstract}
 We show that circular width is preserved under connected sum of knots for some cases.
\end{abstract}

\maketitle

\section{Introduction}
In \cite{MG} the second author defined \textit{circular thin position} and \textit{circular width} for a knot in $S^3$. The idea is to find collections of surfaces  $\{S_i\}_{i=1}^n$ and $\{F_i\}_{i=1}^n$, not necessarily connected, which are properly embedded in the knot exterior, such that  each $F_i$ and each $S_i$ contains a Seifert surface for the knot. When the knot complement is cut open along the collection $\{F_i\}_{i=1}^n$ the result is a collection of disjoint submanifolds whose Heegaard surfaces are the $S_i's$.  We assign a complexity $c(S_i)$ to each $S_i$, and  define the circular width of the exterior of the knot, $cw(E(K))$, as the minimal  ordered $n$-tuple that encodes these complexities.

A decomposition that realizes  the circular width of the knot is called circular thin position of the knot. Circular thin position guarantees that all the $F_i's$ are incompressible and all the $S_i's$ are weakly incompressible. Hence when the knot complement is in circular thin position we obtain a nice sequence of Seifert surfaces which are alternately incompressible and weakly incompressible.

Given two knots $K_1$ and $K_2$ in $S^3$, we can take their connected sum $K_1\sharp K_2$, it is natural to study the behavior of circular width under this operation. In \cite{MG}  an upper  bound for the circular width of $K_1 \sharp K_2$ is given, which depends on the circular width of the original knot exteriors. Namely;

\begin{equation}
cw(E(K_1\sharp K_2)) \leq cw(E(K_1))\sharp cw(E(K_2))
\label{aditivo}
\end{equation}

In this paper we analize knots in $S^3$ having a circular thin position containing a minimal genus Seifert surface and we prove that equality in equation \eqref{aditivo} holds.

Our main result is:

\begin{theo}
\label{teoremon}
Let $K_1$ and $K_2$ be knots in $S^3$. The equation 

$cw(E(K_1\sharp K_2)) = cw(E(K_1))\sharp cw(E(K_2))$

holds for the following cases:

\begin{enumerate}
\item $K_1$ and $K_2$ are fibered knots.
\item $K_1$ is fibered and $K_2$ is not fibered.
\item $K_1$ and $K_2$ are non-fibered knots. $E(K_1)$ and $E(K_2)$ have circular thin positions containing minimal genus Seifert surfaces as a thin level.
\end{enumerate}

\end{theo}

This paper is organized as follows.  In Section \ref{prelim} we give definitions and some facts about knots and Heegaard splittings. Circular thin position is defined in Section \ref{circularthin}, we also discuss the behavior of circular width under connected sum of knots.  In Section \ref{orderedtuples} we study in detail ordered n-tuples, we prove Proposition \ref{ntuples} which is a technical result needed to prove our main theorem.  In Section \ref{additivity} we prove Proposition \ref{prop1} and Corollary \ref{coro1} which allow us to construct a circular handle decomposition for each summand in a connected sum of two knots, we also prove Theorem \ref{teoremon}.

\section{Preliminaries}
\label{prelim}
In this section we begin by briefly recalling some notions for the theory of knots, Seifert surfaces  and Heegaard splittings. 

\subsection{Knots and surfaces}
This section is devoted to  definitions related to knots and Seifert surfaces, as well as to  properties of Seifert surfaces under two operations on knots. The definitions and operations are mostly classical.

Let $K$ be a knot in $S^3$. The knot complement will be denoted by $C_K = S^3 \setminus K$ . An open  tubular neighborhood of $K$  will be denoted by $N(K)$ and the exterior of the knot $K$ by $E(K)=S^3 \setminus N(K)$.

A Seifert surface $R'$ for a knot $K$ is an oriented compact 2-submanifold of $S^3$ with no closed components such that $\bd R' = K$. The intersection of $R'$ with $E(K)$, $R= R' \cap E(K)$, is also called a Seifert surface for $K$.

The \textit{genus} of a knot $K$ is the least genus of all its Seifert surfaces. A surface realizing the genus of a knot is called a \textit{minimal genus Seifer surface}.

Since $R$ is two sided we can specify a $+side$ and a $-side$ of $R$. We say that a disk $D$, such that $\bd D \subset R$, lies on the $+side$ (resp. in the $-side$) of $R$ if the collar of its boundary lies on the $+side$ (resp. in the $-side$) of $R$.

\begin{definition}
Let $S$ be a surface in a 3-manifold $M$. We say that $S$ is \textit{compressible} if there is a  2-disk $D \subset M$ such that $D \cap int(S) = \bd D$ does not bound a disk in $S$.  $D$ is a compressing disk for $S$. If $S$ is not compressible, it is said to be \textit{incompressible}. 

We say that $S$ is \textit{strongly compressible} if there are  two compressing disks,  $D_1$ lying on the +side of $S$ and $D_2$ lying on the $-$side of $S$,  with $\bd D_1$ and $\bd D_2$ disjoint  essential closed curves in $S$. Otherwise we say that $S$ is \textit{weakly incompressible}.
\end{definition}

\begin{definition}
 \label{connected sum}
The \textit{connected sum of two knots} $K_1$ and $K_2$, denoted by $K_1 \sharp K_2$, is constructed by removing a short segment from each $K_i$ and joining each free end of $K_1$ to a different end of $K_2$ to form a new knot. This operation is well-defined up to orientation. 
There is a 2-sphere $\Sigma$ that intersects $K_1\sharp K_2$ in two points and decomposes it in $K_1$ and $K_2$.  $\Sigma$ is called a \textit{decomposing} sphere.

Given Seifert surfaces $S_1$ and $S_2$  for $K_1$ and $K_2$, respectively,  one may construct a Seifert surface for the knot $K_1 \sharp K_2$ by taking a boundary connected sum of $S_1$ and $S_2$, denoted by $S_1 \sharp _{\bd} S_2$. 

\end{definition}

\subsection{Heegaard splittings}
All manifolds will be orientable.

\begin{definition}
A \textit {compression body} $W$ is a cobordism rel $\partial$ between surfaces $\partial_{+}W$ and  $\partial_{-}W$ such that $W \cong \partial_{+}W \times [0,1] \cup$2-handles$\cup$3-handles, where the 2-handles are attached along $\partial_{+}W \times{1}$ and any resulting 2-sphere is capped off with a 3-handle. If $\partial_{-}W \neq \emptyset$ and $W$ is connected, then $W$ is obtained from $\partial_{-}W  \times  I$ by attaching a number of 1-handles along disks on $\partial_{-}W \times \{1\}$ where $\partial_{-}W$ corresponds to $\partial_{-}W \times \{0\}$.

\end{definition}

\begin{figure}[htp]
\centering
\includegraphics[width=5cm]{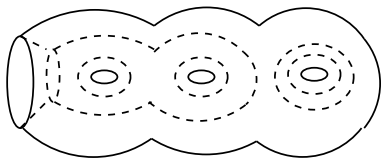}
\caption{A compression body $W$ with $\bd_-W$ a genus 2 surface with one boundary component and a genus 1 surface. $\bd_+W$ is a genus 3 surface with one boundary component}\label{fig.1}
\end{figure}

\begin{definition}
A \textit{3-manifold triad $(M;N,N')$ } is a cobordism $M$ rel $\partial$ between surfaces $N$ and $N'$. Thus $N$ and $N'$ are disjoint surfaces in $\partial M$ with $\partial N \cong \partial N'$ such that $\partial M= N\cup N' \cup (\partial N \times I)$.
\end{definition}

\begin{definition}
A \textit{Heegaard splitting} of $(M;N,N')$ is a pair of compression bodies $(W,W')$ such that $W\cup W'=M$,  $W\cap W'=\partial_{+}W=\partial_{+}W'(=S)$ and $\partial_{-}W=N$, $\partial_{-}W'=N'$. 

$S$ is called a Heegaard surface and $\partial S \cong \partial N$.

The \textit{genus} of a Heegaard splitting is defined by the genus of the Heegaard surface.

A Heegaard splitting $(W,W')$ is said to be \textit{weakly reducible} if there are disks $D_1 \subset W$ and $D_2 \subset W'$ with $\bd D_i \subset S$ an essential curve, for $i=1,2$,  and such that $\bd D_1 \cap \bd D_2 = \emptyset$.

If the Heegaard splitting is not weakly reducible then it is said to be \textit{strongly irreducible}.

\end{definition}

\begin{definition}
A Heegaard splitting $M=H_1 \cup_{S} H_2$ is $\bd$-reducible if there is a $\bd$-reducing disk for $M$ which intersects $S$ in a single curve.
\end{definition}

\begin{prop}
(see \cite{S} Proposition 3.6 )
Any Heegaard  splitting of a $\bd$-reducible 3-manifold is $\bd$-reducible.
\label{boundred}
\end{prop}

\section{Circular thin position}
\label{circularthin}
This was introduced by the second author in \cite{MG}. For sake of completeness we include some definitions and results.

Given a regular circled-valued  Morse function on the complement of a knot $C_K$,  $f: C_K \rightarrow S^1$, as in the case of real-valued Morse functions, there is a correspondence between $f$ and a handle decomposition for $E(K)$, namely 
\begin{center}
$E(K)= (R \times I)\cup N_1 \cup T_1 \cup N_2 \cup T_2 \cup...\cup N_k \cup T_k \cup b_3 /( R\times 0 \sim R\times 1)$, 
\end{center}

\noindent where $R$ is a  Seifert surface for $K$, $R\setminus K$ is a regular level surface of $f$, $N_i$ is a collection of 1-handles corresponding to index 1 critical points, $T_i$ is a collection of 2-handles corresponding to index 2 critical points and $b_3$ is a collection of 3-handles.

We will call this decomposition a \textit{circular handle decomposition} for $E(K)$.

Let us denote by $S_i$ the surface  $cl (\partial (( R \times I)\cup N_1\cup T_1...\cup N_i )\setminus \bd E(K) \setminus R \times 0)$ and let $F_{i+1}$ be the surface $cl(\partial (( R \times I)\cup N_1\cup T_1...\cup T_i )\setminus \partial E(K) \setminus R \times 0)$, where $cl$  means the closure. When $i=k$, $F_{k+1}= F_1= R$. Every $S_i$ and $F_i$ contains a Seifert surface for $K$; note that $F_i$ or $S_i$ may be disconnected.

The surfaces $S_i$ and $F_i$, for $i=1,2,...,k$ will be called \textit{level surfaces}.

A level surface $F_i$ is called a \textit{thin surface} and a level surface $S_i$ is called a \textit{thick surface}. 

Let  $W_i=($collar of $F_i)\cup N_i\cup T_i$. $W_i$ is divided by a copy of $S_i$ into two compression bodies $A_i=($collar of $F_i)\cup N_i$ and $B_i=($collar of $S_i)\cup T_i$. Thus $S_i$ describes a Heegaard splitting of $W_i$ into compression bodies $A_i$ and $B_i$, where  $\partial_{-}A_1=R$, $\partial_{+} A_i= \partial_{+}B_i=S_i$, $\bd_{-}B_i=\bd_{-}A_{i+1}=F_{i+1}$ ($i=1,2,...,k-1$), $\partial_{-}B_k=R$. Thus we can write

$E(K)= A_1 \cup _ {S_1} B_1 \bigcup _ {F_2} A_2 \cup _ {S_2} B_2 \bigcup _ {F_3} ... \bigcup _ {F_{k}} A_k \cup _ {S_k} B_k$.

Figure  \ref{cirdec} shows a schematic picture of a circular handle decomposition with level surfaces and compression bodies indicated.

\begin{figure}[htp]
\labellist 
\small\hair 2pt 
\pinlabel $F_1$ at 183 88
\pinlabel $S_1$ at 153 157
\pinlabel $F_2$ at 87 184
\pinlabel $S_2$ at 19 156
\pinlabel $F_3$ at -9 87
\pinlabel $S_3$ at 17 26
\pinlabel $F_4$ at 94 -7
\pinlabel $S_4$ at 149 15
\pinlabel $A_1$ at 217 145
\pinlabel $B_1$ at 141 204
\pinlabel $W_1$ at 220 208
\pinlabel $N_1$ at 137 107
\pinlabel $T_1$  at 105 141
\pinlabel $N_2$ at 66 145 
\pinlabel $T_2$  at 30 109 
\pinlabel $N_3$ at 29 68
\pinlabel $T_3$  at 63 33
\pinlabel $N_4$ at 106 34
\pinlabel $T_4$ at 140 67 
\endlabellist 

\centering
\includegraphics[width=5cm]{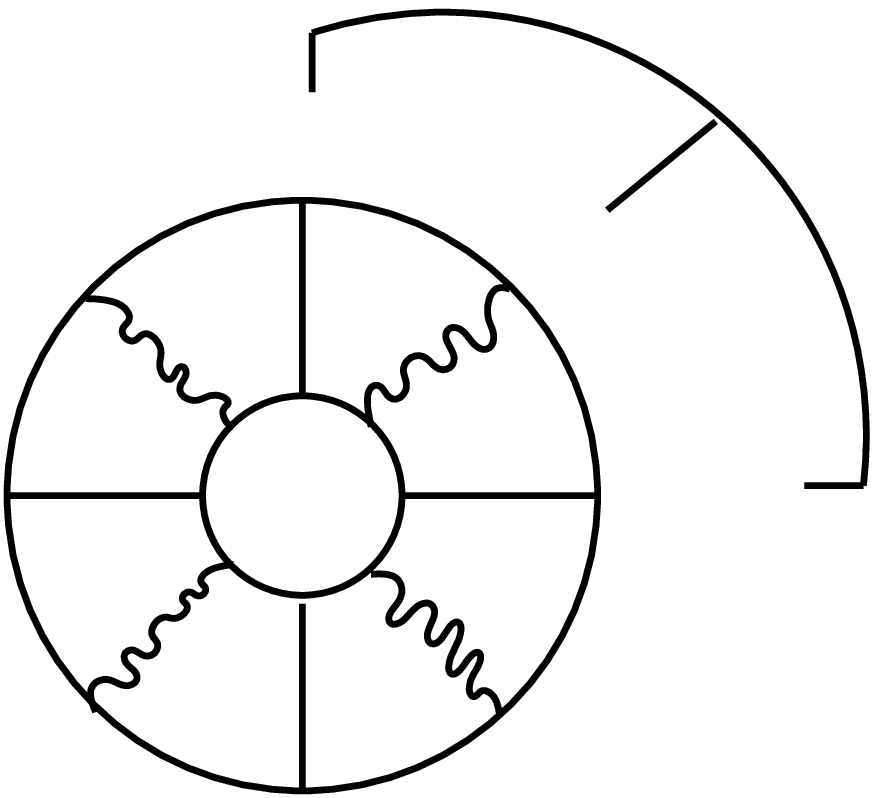}
\caption { Splitting of $E(K)$ into compression bodies}\label{cirdec}
\end{figure}

We wish to find a decomposition in which the $S_i$ are as simple as possible. 

\begin{definition}
For a compact connected surface $G$ different from $S^2$ or $D^2$ define the complexity of $G$, $c(G)$, to be  $c(G)=1-\chi(G)$.  If $G=S^2$ or $G=D^2$, set $c(G)=0$. If $G$  is disconnected we define $c(G)=\Sigma (c(G_i))$ where $G_i$ are the components of $G$.

Let $K$ be a knot in $S^3$.
Let $D$ be a circular handle decomposition for $E(K)$. Define the \textit{circular width of $E(K)$ with respect to the decomposition D , $cw(E(K),D)$}, to be the set of integers $\{ c(S_i),  1\leq i \leq k \} $. Arrange each multi-set  of integers in monotonically non-increasing order, and then  compare the ordered multisets lexicographically.

The \textit{circular width of $E(K)$}, denoted $cw(E(K))$, is the minimal circular width, $cw(E(K),D)$  over all possible circular decompositions $D$ for $E(K)$.

$E(K)$ is in \textit{circular thin position} if the circular width of the decomposition is the circular width of $E(K)$. 

If a knot $K$ is fibered we define the circular width of $K$, $cw(E(K))$, to be equal to zero.
\end{definition}

A nice property of a knot in circular thin position is that the thin surfaces are incompressible and the thick surfaces are weakly incompressible. For a proof of this fact see Theorem 3.2, \cite{MG}. 

\begin{definition}
 A circular handle decomposition $D$ for a knot exterior $E(K)$ is called a \textit{ circular locally thin} decomposition if the thin level surfaces  $F_i$'s are incompressible and the thick level surfaces $S_i$'s are weakly incompressible.
\end{definition}

\begin{definition}
$K$ is \textit{almost fibered} if there is a Seifert surface $R$ so that $E(K)$ has a circular thin decomposition of the form 

$E(K)=(R \times I) \cup N_1 \cup T_1 / (R\times 0 \sim R \times 1)$.
\end{definition}

Figure \ref{esquema_almost_fib} shows a schematic picture of  an almost fibered knot.

\begin{figure}[htp]
\labellist 
\small\hair 2pt 
\pinlabel $R$ at 234 144
\pinlabel $N_1$ at 131 228
\pinlabel $S$ at 43 149
\pinlabel $T_1$ at 139 34
\endlabellist
\centering
\includegraphics[width=4cm]{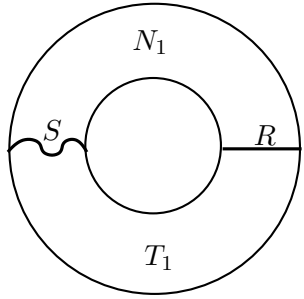}
\caption{An almost fibered knot.}
\label{esquema_almost_fib}
\end{figure}

\subsection{Behavior of circular width under connected sum}

Let us consider the knot exteriors $E(K_1)$ and $E(K_2)$. Assume they have the following circular handle decompositions:

 \begin{center}
 $E(K_1)= (F_1 \times I) \cup N_1 \cup T_1 \cup N_2 \cup T_2 \cup ... \cup N_n \cup T_n  \cup b_3^1/ (F_1\times 0 \sim F_1 \times 1)$ 
 \end{center} 

with level surfaces $F_1$, $G_1$, $F_2$..., $F_n$, $G_n$.

\begin{center}
$E(K_2)= (R_1 \times I )\cup O_1 \cup W_1 \cup O_2 \cup W_2 \cup ... \cup O_m \cup W_m \cup b_3^2/ (R_1 \times 0 \sim R_1 \times 1)$
\end{center}

with level surfaces  $R_1$, $S_1$, $R_2$...,$R_m$, $S_m$.

Let $K= K_1 \sharp K_2$. There is a natural way to obtain a circular handle decomposition for $E(K)$ as follows. Starting with the Seifert surface $R=F_1 \sharp _{\bd}R_1$ for $K$,  we attach the sequence of handles corresponding to $E(K_1)$, i.e., we attach $N_i$ and $T_i$, along the $F_1$ summand of $R$. Then we attach the sequence of handles corresponding to  $E(K_2)$, i.e., we attach  $O_j$ and $W_j$, along the $R_1$ component of $R$.  Notice that this process can be done if we choose different thin surfaces. Thus $K_1\sharp K_2$ inherits $n\times m$ circular handle decompositions each with $n+m$ thin levels and thick levels.

The thin levels for $K= K_1 \sharp K_2$ are homeomorphic to  $\{F_{i_0} \sharp R_j\} \cup \{F_i \sharp R_{j_0}\}$ and the thick levels are homeomorphic to $\{F_{i_0} \sharp S_j\} \cup \{G_i \sharp R_{j_0}\}$, for a fixed $i_0 \in \{1,2,...,n\}$ and $j_0 \in \{1,2,...,m\}$.

Figure \ref{ind_han_dec} is a schematic picture of the induced circular handle decomposition in a complement of a connected sum of two knots. 

\begin{figure}[htp]
\labellist 
\small\hair 2pt
\pinlabel {$E(K_1)$} at  34 543 
\pinlabel $F_1$ at 279 686 
\pinlabel $G_1$ at  127 832
\pinlabel $F_2$  at -15 684
\pinlabel $G_2$ at 134 528
\pinlabel $N_1$  at 189 740 
\pinlabel $T_1$ at 60 747 
\pinlabel $N_2$  at 74 625 
\pinlabel $T_2$ at 194 635 
\pinlabel $E(K_2)$ at 580 538
\pinlabel $R_1$ at 645 685
\pinlabel $S_1$ at 350 689
\pinlabel $O$ at 493 768
\pinlabel $W$ at 493 599
\pinlabel {$E(K_1 \sharp K_2)$} at 295 -18
\pinlabel {$\Gamma= S_1 \sharp_{\bd} F_1$} at 510 41
\pinlabel {$R_1 \sharp_{\bd} F_1$} at 550 196
\pinlabel {$Q_1= R_1 \sharp_{\bd} G_1$} at 500 360
\pinlabel {$\Sigma_1= R_1 \sharp_{\bd} F_2$} at 100 360
\pinlabel {$Q_2= R_1 \sharp_{\bd} G_2$} at 20 196
\pinlabel {$\Sigma_2= R_1 \sharp_{\bd} F_1$} at 100 35
\pinlabel $O$ at 305 113
\pinlabel $W$ at 380 176
\pinlabel $N_1$ at 442 268
\pinlabel $T_1$ at 305 350
\pinlabel $N_2$ at 160 270
\pinlabel $T_2$ at 152 138
\endlabellist
\centering
\includegraphics[width=8cm]{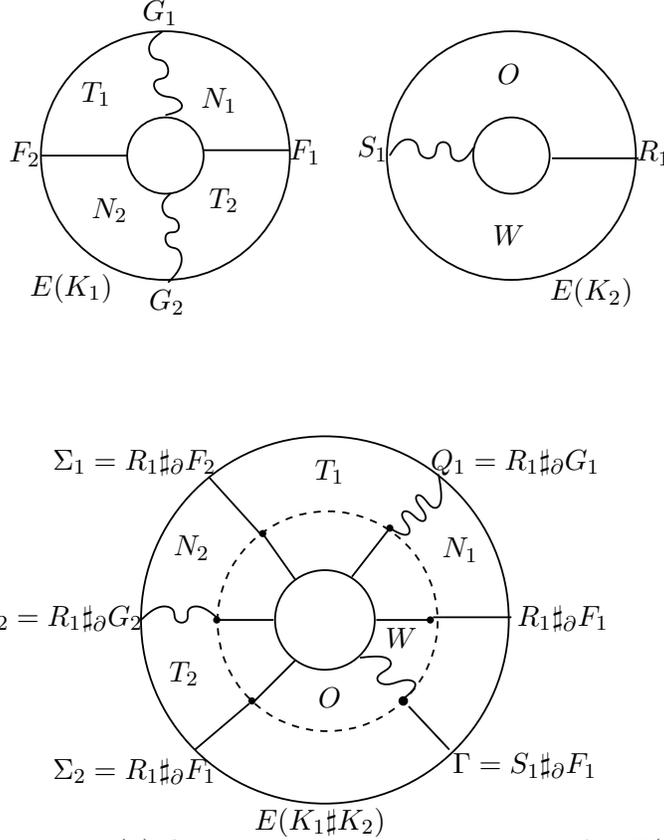}
\caption{(a) Circular handle decomposition for $E(K_1)$, (b) circular handle decomposition for $E(K_2)$ and (c) induced circular handle decomposition for $E(K_1 \sharp K_2)$.}\label{ind_han_dec}
\end{figure}

Since the Euler characteristic for the boundary connected sum equals 
$\chi(S_1\sharp_{\bd}S_2)= \chi(S_1)+\chi(S_2)-1$, then the complexity $c(S)= 1-\chi(S)$ applied to a boundary connected sum $S_1 \sharp_{\bd} S_2$ becomes equal to $c(S_1\sharp_{\bd}S_2)= c(S_1)+c(S_2)$.
  
Each  decomposition for  $E(K_1\sharp K_2)$ has circular width:

\begin{IEEEeqnarray}{rCl} 
cw_D(E(K_1\sharp K_2))= \{c(F_{i_0} \sharp S_1),...,c(F_{i_0} \sharp S_m), c(G_1 \sharp R_{j_0}), .... c(G_n\sharp R_{j_0})\} \nonumber
\end{IEEEeqnarray} 
modulo non-increasing order.

If we choose  $F_{i_0}$ to be a thin level Seifert surface for $K_1$ such that $c(F_{i_0})\leq c(F_i)$ for all $i=1,2,...,n$ and $R_{i_0}$  be a thin level  Seifert surface for $K_2$ such that $c(R_{j_0})\leq c(R_j)$ for all $j=1,2,...,m$, then the decomposition $D$ for $E(K_1\sharp K_2)$ containing $F_{i_0}$ and $R_{j_0}$ as summands of the thick levels will be the one with the smallest circular width amongst all the $n\times m$ circular decompositions.

Let us denote the circular width of such decomposition by:

\begin{equation*}
cw(E(K_1))\sharp cw(E(K_2))
\end{equation*}

which is an upper bound for  $cw(E(K_1\sharp K_2))$. Thus we have:

\begin{equation}
cw(E(K_1\sharp K_2)) \leq cw(E(K_1))\sharp cw(E(K_2))
\label{aditivo2}
\end{equation}

Moreover, in \cite{MG} it is proved  that if $K_1$ and $K_2$ are in circular thin decomposition, the circular handle decomposition induced on $K_1\sharp K_2$ is circular locally thin. Thus, is natural to ask if such decomposition is the thinnest for $K_1\sharp K_2$ and if equality in  \eqref{aditivo2} holds.

\section{Ordered n-tuples}
\label{orderedtuples}

We have defined the circular width as an ordered $n$-tuple, say $a=(a_1, a_2, ..., a_n)$ where $a_1 \geq a_2 \geq... \geq a_n$. For simplicity it will be called  just and $n$-tuple.

We can compare a $m$-tuple and a $n$-tuple using the lexicographic order. More precisely, we have the following definition.

\begin{definition}
Let  $a=(a_1, a_2, ..., a_n)$ be a $n$-tuple and let   $b=(b_1, b_2, ..., b_m)$ be  a $m$-tuple. We say that;
\begin{enumerate}
\item $a=b$ if and only if $m=n$ and $a_i=b_i$ for all $i$.  
\item $a<b$
\begin{enumerate}
\item If there exists $i_o$ such that $a_{i_o}< b_{i_o}$ and $a_i=b_i$ for all $i<i_o$, or 
\item If $n<m$ and $a_i=b_i$ for all $1\leq i \leq n$. 
 \end{enumerate}
\end{enumerate}
\end{definition}

\begin{remark}If $a$ is a $n$-tuple and $b$ is a $m$-tuple such that  $a\leq b$ and $\alpha$ a non-negative real number, then 
 $(a_1+ \alpha, ... , a_k+ \alpha ) \leq (b_1+ \alpha, ... , b_l+ \alpha)$.
\end{remark}

Given a $n$-tuple $a$ and a $m$-tuple $b$ we can define a new $(n+m)$-tuple as follows:

\begin{definition}
Let  $a$ be a $n$-tuple and $b$ be a $m$-tuple. Define the union of $a$ and $b$, denoted by $a\cup b$, as the $(n+m)$-tuple whose entries are all the elements of $\{a_1, a_2, ..., a_n, b_1, ..b_m\}$, ordered in non-increasing order. 

\noindent For instance if $a=(4,4,3,3,1,1,1)$ and $b=(7,7,5,5,5,3,3,3,1,1,1,1)$ then $a\cup b= (7,7,5,5,5,4,4,3,3,3,3,3,1,1,1,1,1,1, 1)$. 

\end{definition}

The following result is used in the proof of Theorem \ref{bigteo}:

\begin{prop}
\label{ntuples}
Let  $a$, $b$, $c$ and $d$ tuples such that $b\leq a$  and  $d\leq c$. Then $b\cup d \leq a\cup c$.
\end{prop}
\begin{proof}
Case 1: If $a=b=c=d$, it is easy to verify that $a\cup c = b\cup d$.

Case 2: If $a=b$ and $d<c$.

 Subcase 2.1: There is an index $j_0$ such that $d_{j_0}< c_{j_0}$ and $d_s=c_s$ for all $s<j_0$.

Suppose there is $l_0$ such that $a_{l_0} < c_{j_0} \leq a_{l_0-1}$.  Then the $(l_0 + j_0 -2)$th-entry for both $a\cup c$ and $b\cup d$ coincide. The $(l_0+j_0-1)$th-entry for $a\cup c$ is either $c_{j_0}$ or $a_{l_0}$, by assumption $a_{l_0}< c_{j_0}$, thus it must be $c_{j_0}$. On the other hand the $(l_0+j_0-1)$th-entry for $b\cup d$ is chosen from $d_{j_0}$ and $a_{l_0}$, in either case both are strictly smaller than $c_{j_0}$, therefore $b\cup d < a\cup c$.

Suppose that $c_{j_0} < a_k$ for all $k$. If there is $l_0$ such that $a_{l_0-1}=c_{j_0-1}> a_{l_0}$, then
the  $(l_0 + j_0 -2)$th-entry for both $a\cup c$ and $b\cup d$ coincide.  The $(l_0+j_0-1)$th-entry for $a\cup c$ is either $a_{l_0}$ or $c_{j_0}$, by assumption $c_{j_0}< a_k$ for all $k$ then we must choose $a_{l_0}$, the next entry is $a_{l_0+1}$, and so on until the entry is $a_n$, then the entry that follows must be $c_{j_0}$. Similarly happens for $b\cup d$, its $(l_0+j_0-1)$th-entry is $a_{l_0}$, the next one is $a_{l_0+1}$, and so on until the entry is $a_n$, then the next entry is $d_{j_0}$ which is strictly smaller that $c_{j_0}$, thus $b\cup d < a\cup c$.

Subcase 2.2: $d$ is a $n$-tuple and $c$ is a $m$-tuple such that $n<m$ and $d_i=c_i$ for all $1\leq i \leq n$.  Suppose $a=b$ is a $k$-tuple.

If  $a_j\geq d_n$ for all $1\leq j \leq k$. Then $a\cup c$ is a $(m+k)$-tuple and $b \cup d$ is a $(n+k)$-tuple such that $n+k < m+k$ and the entries of $a\cup c$ and $b \cup d$ coincide up to the $(n+k)$th-entry. Thus $b\cup d < a \cup c$.

If there is $l_0$ such that $a_{l_0} < d_n \leq a_{l_0-1}$. Then $a\cup c$ and $b\cup d$ coincide up to the $(n+l_0-1)$th-entry which is equal to $c_n=d_n$. The remainder entries for $b\cup d$ are $a_{l_0}, a_{l_0+1}, ... , a_k$ in that order. On the other hand the remainder entries for $a\cup c$ are taken from $\{c_s , n< s \leq m\}$ and $\{a_t, l_0 \leq t \leq k \}$. Then either $a\cup c$ and $b\cup d$ are equal up to the (n+k)th-entry, or there is a $u_0 > n+l_0-1$ such that $x_{u_0}< y_{u_0}$ where $x_{u_0}$ is an entry of $b \cup d$ and $y_{u_0}$ is an entry for $a \cup c$, which imply that $b\cup d < a\cup c$.

Case 3: If $b <a$ and $d< c$. Using case 2, we have that $b \cup d < a \cup d$ and that $d\cup a < c \cup a$. These two inequalities imply $b\cup d < a \cup c$.
\end{proof}

\section{Additivity of circular width under connected sum}
\label{additivity}
First we need to prove that a circular (locally) thin handle decomposition for $E(K_1\sharp K_2)$ induces a circular handle decomposition on each summand $E(K_1)$ and $E(K_2)$.

Recall that for a connected sum of knots, $K_1\sharp K_2$, there is a decomposing sphere $\Sigma$ that intersects    $K_1\sharp K_2$ in two points. Let $A$ be the annulus in $E(K_1 \sharp K_2)$ given by $\Sigma \cap E(K_1\sharp K_2)$. 

The following proposition shows that $A$ intersects the collection of thin and thick surfaces for $E(K_1\sharp K_2)$ in essential arcs.
\begin{prop}
\label{prop1}
Suppose that $E(K_1 \sharp K_2)$ is in circular (locally) thin position  with $\F$ the family of thin surfaces and $\Su$ the family of thick surfaces. Then $\F \cup \Su$ can be isotoped to intersect $A$ only in arcs that are essential in both $A$ and $\F \cup \Su$.
\end{prop}

\begin{proof}
The annulus $A$ is  properly embedded in $E(K_1 \sharp K_2)$, its boundary components are meridian disks in $\bd E(K_1 \sharp K_2)$.  We arrange $A$ and $\F \cup \Su$ to be transverse and conclude that $A$ intersects each $F_i \in \F$ and each $ S_i \in \Su$ in exactly one arc (properly embedded and essential in $A$) and a finite number of simple closed curves.
We need to remove this later curves.

Let $F_i \in \F$, since $F_i$ is incompressible and $E(K_1\sharp K_2)$ is irreducible then, using an innermost disk argument,  $A \cap F_i$ does not contain closed curves. Thus $F_i \cap A$ consists of a single  properly embedded separating essential arc  in $F_i$.

All curves in  $A\cap S_i$ are essential  in $S_i$, otherwise using the irreducibility of $E(K_1 \sharp K_2)$ we get rid of of  inessential curves.

Each $S_i$ determines a Heegaard splitting given by $A_i \cup_{S_i} B_i$, with $\bd_{-}A_i = F_i$, $\bd_{-}B_i = F_{i+1}$ and $\bd_{+}A_i = \bd_{+}B_i=S_i$. Let $R_i$ be the region on $A$ cobounded by the arcs $\alpha_i= A\cap F_i$ and $\alpha_{i+1}=A \cap F_{i+1}$.  $R_i$ contains an arc $\beta_i$ and simple closed curves contained in $A\cap S_i$.

A disk of $R_i-S_i$ compresses $S_i$ in one of the two compression bodies $A_i$ or $B_i$, say $A_i$. Since $S_i$ is weakly incompressible, all disks components of $R_i-S_i$ lie in $A_i$.

\begin{claim}
The curves in $R_i \cap S_i$ are non nested in $R_i$.
\end{claim}

If any pair of curves of $R_i \cap S_i$ are nested (they are inessential in $R_i$) then the outer curve of the innermost such pair cuts off a component $C$ of $R_i-S_i$ so that all but one of the curves in $\bd C$ are adjacent to disks in $A_i$ (thus $C \subset B_i$) and precisely one, denoted by $\gamma$, is not.
Compress $S_i$ into $A_i$ along 2-handles whose cores are the disks with boundaries on $\bd C$. Let $\bar{S_i}$ be the result of this compression. Let $\bar{B_i}$ be the 3-manifold obtained from $B_i$ by attaching these 2-handles to $B_i$. Thus $S_i$ determines a Heegaard splitting for $\bar{B_i}$, namely $\bar{B_i}= B_i \cup_{S_i} B_i'$, where $B_i'=(S_i \times I) \cup$2-handles, $\bd_{-}B_i= F_{i+1}$ and $\bd_{-}B_i'=\bar{S_i}$. A copy of the curve $\gamma$ lies in $\bar{S_i}$ and it  is the boundary of a disk $D$ in $\bar{B_i}$. Suppose that $\gamma$ is non-trivial in $\bar{S_i}$ so $D$ is a $\bd$-reducing disk for $\bar{B_i}$. Then the Heegaard splitting $\bar{B_i}= B_i \cup_{S_i}B_i'$ is $\bd$-reducible, by Proposition \ref{boundred} there is a $\bd$-reducing disk $D'$ for $\bar{B_i}$ that intersects $S_i$ in a single curve $\gamma '$. Moreover $\bd D' \subset \bar{S_i}$.

Observe that $D'$ intersects $B_i'$ in an annulus $A'$ with one boundary component on $\bar{S_i}$ and the other one on $S_i$  is $\gamma'$. $D'$ intersects $B_i$ in a disk $D''$ with $\bd D''=\gamma'$. The annulus $A'$ is a product annulus and it is contained in the region homeomorphic to $S_i \times I$. Thus the boundary components of $A'$ are disjoint from the  cores of the 2-handles attached to $S_i$. 
In particular the boundary $\gamma'$ of $D''$ is disjoint from the cores of the 2-handles attached to $S_i$.  Then $D''$ is a compression disk for $S_i$ contained in $B_i$ whose boundary is disjoint from a set of compressing disks contained in $A_i$, this fact contradicts the weakly compressibility of $S_i$. Therefore $\gamma$ must bound a disk in $\bar{S_i}$. Push the disk $\gamma$ bounds in $\bar{S_i}$ slightly into $A_i$, this is a disk $D$ in $A_i$ whose boundary is parallel to $\gamma$ in the component  of $R_i$ adjacent to $C$ across $\gamma$. Replacing the subdisk of $R_i$ bounded by $\gamma$ by the disk $D$ allows us to remove the nested curves.

Thus $S_i \cap R_i$ contains one arc and non nested curves.

Let $R_{A_i}$ denote $R_i \cap A_i$ and  $R_{B_i}$ denote $R_i \cap B_i$. In $R_{B_i}$ there are non nested closed curves that bound disks in $A_i$. We can assume that $R_{A_i} \cap S_i$ is empty, otherwise we must see nested curves.

$R_{B_i}$ is a planar surface contained in $B_i$. $R_{B_i}$ is incompressible, for otherwise, by doing a compression we get a surface $A_i'$ isotopic to $A_i$ with fewer intersections with $S_i$. However $R_{B_i}$ must be $\bd$-compressible. This can be seen by looking at the intersections of $R_{B_i}$ with a collection of meridian disks and spanning annuli in $B_i$. There are 4 types of $\bd$-compressions, determined by the types of arcs shown in Figure \ref{tipo1}.
\begin{figure}[htp]
\labellist 
\small\hair 2pt 

\pinlabel $R_{A_i}$ at 215 150
\pinlabel $R_{B_i}$ at 205 300
\pinlabel $1$ at 150 200
\pinlabel $2$ at 350 265
\pinlabel $3$ at 210 210
\pinlabel $4$ at 332 300
\endlabellist 
\centering
\includegraphics[width=8cm]{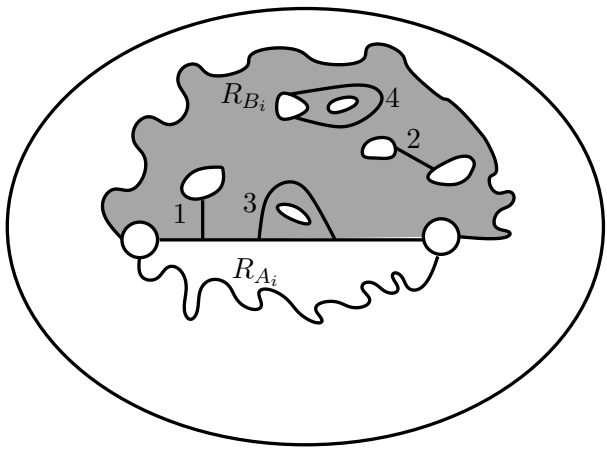}
\caption{Four types of arcs} \label{tipo1}
\end{figure}

If $\Delta$ is a $\bd$-compressing disk for $R_{B_i}$ where $\bd \Delta = \delta_1 \cup \delta_2$, $\delta_1 \subset R_{B_i}$, $\delta_2 \subset S_i$, then a boundary compression along $\Delta$ pushes a regular neighborhood of $\delta_1$ into $A_i$. After performing  boundary compressions corresponding to arcs of type 1 and 2  the number of curves of intersection is decreased by 1. By performing boundary compressions corresponding to arcs  of type 3 and 4, nested curves are generated, which can be removed. See Figure \ref{tipo2}.

\begin{figure}[htp]
\labellist 
\small\hair 2pt 
\pinlabel $R_{A_i}$ at 215 150
\pinlabel $R_{B_i}$ at 184 300
\endlabellist 
\centering
\includegraphics[width=8cm]{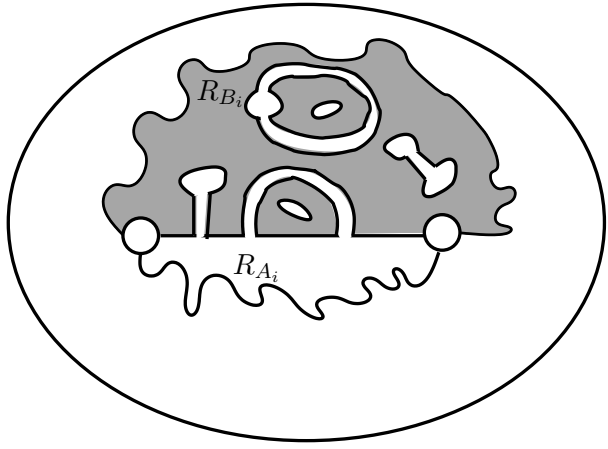}
\caption{$R_{B_i}$ after boundary compressing}\label{tipo2}
\end{figure}

Thus $R_i \cap S_i$ contains only one arc. Therefore the annulus $A$ intersects each $S_i \in \Su$ in one arc.

\end{proof}

This proposition allows us to push 1-handles and 2-handles away from the annulus $A$. Moreover a collection of 1-handles $N_i$ (or a collection of 2-handles $T_i$) can be pushed away from $A$ in such a way that $N_i$ (or $T_i$) is totally contained in $E(K_j)\cap E(K_1\sharp K_2)$, for some $j=1,2$.

In other words, a circular (locally) thin decomposition for $E(K_1\sharp K_2)$ induces circular locally thin decompositions for $E(K_1)$ and $E(K_2)$.

\begin{coro}
\label{coro1}
Suppose  $K= K_1 \sharp K_2$ is  in circular (locally) thin position. Let $E(K)= (F \times I) \cup N_1 \cup T_1 \cup ... \cup N_m \cup T_m / F \times {-1} \sim F \times{1}$
be a handle decomposition realizing a circular (locally) thin position. Let $\mathcal{N}$ be the collection of $N_i$'s, let $\mathcal{T}$ be the collection of $T_i$'s.  Then there  are subcollections $\mathcal{N}_1$ and $\mathcal{N}_2$ of $\mathcal{N}$ such that $\mathcal{N}_1 \cup \mathcal{N}_2 = \mathcal{N}$ and $\mathcal{N}_1 \cap \mathcal{N}_2 = \emptyset$, and  subcollections $\mathcal{T}_1$ and $\mathcal{T}_2$ of $\mathcal{T}$ such that $\mathcal{T}_1 \cup \mathcal{T}_2 = \mathcal{T}$ and $\mathcal{T}_1 \cap \mathcal{T}_2 = \emptyset$, such that $\mathcal{N}_i$ and $\mathcal{T}_i$ define a circular handle decomposition for $E(K_i)$, $i=1,2$.
\end{coro}

\begin{proof}
By Proposition \ref{prop1} the annulus $A= \Sigma \cap E(K_1 \sharp K_2)$ intersects each $F_i$ and each $S_i$ in a single arc. Each arc $A \cap F_i$ separates $F_i$  in $F_i'$(contained in $E(K_1)$) and $F_i''$ (contained in $E(K_2)$), and each arc $A\cap S_i$ separates $S_i$ in $S_i'$ (contained in $E(K_1)$)  and $S_i''$ (contained in $E(K_2)$).  If one 1-handle of $N_i$ is attached to $F_i'\times I$ and a 2-handle of $T_i$ is attached to $S_i''\times I$ (or viceversa), then  these handles determine compressing disks for $S_i$ which are disjoint and lie on opposite sides of $S_i$, which contradicts that $S_i$ is weakly incompressible. Therefore the collection of 1-handles $N_i$ must be attached either to $F_i'\times I$ or to $F_i''\times I$, say $F'_i \times I$, and then the collection of 2-handles $T_i$ is attached along $S_i'$.

Since $F_i \simeq F_i' \sharp F_i''$, if the collection $N_i$ has been attached along $F_i'$ (or $F_i''$) then the surface $S_i$ is homeomorphic to $S_i' \sharp F_i''$ (or $F'_i \sharp S_i''$).

In general we will see the following: Begin with the surface $F_1=F\simeq F_1' \sharp F_1''$, we will see a subcollection $\mathcal{N}_j^1$ of $\mathcal{N}$ and a subcollection $\mathcal{T}_j^1$ of                   $\mathcal{T}$ contained in $E(K_j)$, say $j=1$. In other words there is $i_0$ ($1\leq i_0 \leq m$) such that for every $1 \leq i \leq i_0$, the handles $N_i$ are attached along $F_i'$ and the handles $T_i$ are attached along $S_i'$.

If $i_0=m$ then  $\mathcal{N}$ and $\mathcal{T}$ happened to be contained, say in $E(K_1)$, then the knot $K_2$ is fibered and $\mathcal{N}_1=\mathcal{N}$, $\mathcal{N}_2=\emptyset$, $\mathcal{T}_1=\mathcal{T}$, $\mathcal{T}_2=\emptyset$.

If $i_0 < m$ then there is a subcollection $\mathcal{N}_2^1$ of $\mathcal{N}- \mathcal{N}_1^1$ and a subcollection $\mathcal{T}_2^1$ of $\mathcal{T}-\mathcal{T}_1^1$ contained in $E(K_2)$. In other words there is $i_1$ ($i_0+1\leq i_1 \leq m$) such that for every $i_0+1 \leq i \leq i_1$, the handles $N_i$ are attached along $F_i''$ and the handles $T_i$ are attached along $S_i''$.

If $i_1 < m$ then there is a subcollection $\mathcal{N}_1^2$ of $\mathcal{N}- (\mathcal{N}_1^1 \cup \mathcal{N}_2^1)$ and a subcollection $\mathcal{T}_1^2$ of $\mathcal{T}-(\mathcal{T}_1^1\cup \mathcal{T}_2^1)$ contained in $E(K_1)$. In other words there is $i_2$ ($i_1+1\leq i_2 \leq m$) such that for every $i_1+1 \leq i \leq i_2$, the handles $N_i$ are attached along $F_i'$ and the handles $T_i$ are attached along $S_i'$.

We conclude when $i_s=m$, then we have a subcollection $\mathcal{N}_j^s$ of $\mathcal{N}-( \mathcal{N}_1^1\cup \mathcal{N}_2^1 \cup \mathcal{N}_1^2 \cup... \cup \mathcal{N}_{j' \neq j}^{s-1})$ and a subcollection $\mathcal{T}_j^s$ of $\mathcal{T}-( \mathcal{T}_1^1\cup \mathcal{T}_2^1 \cup \mathcal{T}_1^2 \cup... \cup \mathcal{T}_{j'}^{s-1})$ contained in $E(K_j)$, where  $j,j'\in\{1,2\}$ and $j\neq j'$.

Thus the subcollection $\mathcal{N}_j$ is given by $\cup \mathcal{N}_j^k$ and the subcollection $\mathcal{T}_j$ is given by $\cup \mathcal{T}_j^k$, for $j=1,2$.  This proves the corollary.
 \end{proof}

\begin{remark} 
It is not hard to see that we can rearrange the collections of 1-handles and 2-handles in such a way that we first glue all handles contained in one summand, say $E(K_1)$, and then all the handles contained in $E(K_2)$.  
\end{remark}

The following result is an  immediate consequence:

\begin{coro}
If $K=K_1 \sharp K_2$ is almost fibered then either $K_1$ or $K_2$ is fibered, say $K_1$, and $K_2$ is  not fibered.
\end{coro}

Now we are ready to prove:

\begin{theo}
\label{bigteo}
Let $K_1$ and $K_2$ be knots in $S^3$. The equation 

$cw(E(K_1\sharp K_2)) = cw(E(K_1))\sharp cw(E(K_2))$

holds for the following cases:

\begin{enumerate}
\item $K_1$ and $K_2$ are fibered knots.
\item $K_1$ is fibered and $K_2$ is not fibered.
\item $K_1$ and $K_2$ are non-fibered knots. $E(K_1)$ and $E(K_2)$ have circular thin positions containing minimal genus Seifert surfaces as a thin level.
\end{enumerate}

\end{theo}
\begin{proof}
(1) Easily follows from the well known fact that connected sum of two fibered knots is fibered.

(2) Let $F$ be the fiber for $E(K_1)$ and assume $E(K_2)$ is in  circular thin position with $\{R_i\}_1^n$ the collection of thin levels and $\{ S_i\}_1^n$ the collection of thick levels.  Then $E(K_1\sharp K_2)$ inherits a circular handle decomposition with thin levels homeomorphic to the collection $\{F \sharp R_i \}_1^n$  and thick levels homeomorphic to $\{F \sharp S_i \}$, such decomposition has circular width given by $cw(E(K), D)=\{c(F\sharp S_i)\}_1^n$ modulo non-increasing order. 

Suppose that $E(K_1\sharp K_2)$ has a circular a circular thin decomposition with thin levels $\{T_j \}_1^m$ and thick levels $\{U_j\}_1^m$. Proposition \ref{prop1} together with the assumption that $K_1$ is fibered imply that $T_j= F \sharp T'_j$ and $U_j=F\sharp U_j'$, inducing a circular handle decomposition on $E(K_2)$ with thick levels $\{U'_j\}$.  Such decomposition has circular width $cw(E(K_2), D')=\{c(U'_j\}_1^m$ modulo non-increasing order. Thus we have the following inequality:

\begin{equation}
\{c(S_i)\}\leq \{c(U_j')\}
\label{eq1}
\end{equation}
modulo non-increasing order.

If we add $c(F)$ to both sides of equation \eqref{eq1} we obtain:

\begin{displaymath}
\{ c(S_i)+c(F)\}=\{c(F\sharp S_i) \} \leq \{c(U_j')+c(F) \}=\{c(U_j) \}
\end{displaymath}
modulo non-increasing order, which is equivalent to:

\begin{displaymath}
cw(E(K_1))\sharp cw(E(K_2))\leq cw(E(K_1\sharp K_2))
\end{displaymath}.

(3) Let $D_1$ be a circular handle decomposition for $E(K_1)$ which realizes $cw(E(K_1))$. Let $\{ F_i\}_{i=1}^k$ be the collection of thin levels and $\{G_i\}_{i=1}^k$ be the collection of thick levels for $D_1$. Then $cw(E(K_1))=\{c(G_i) \}_{i=1}^k$ modulo non-increasing order.

Let $D_2$ be a circular handle decomposition for $E(K_2)$ which realizes $cw(E(K_2))$. Let $\{ R_j\}_{j=1}^l$ be the collection of thin levels and $\{S_j\}_{j=1}^l$ be the collection of thick levels for $D_2$. Then $cw(E(K_2))=\{c(S_j) \}_{j=1}^l$ modulo non-increasing order.

Assume that $F_1$ and $R_1$ are minimal genus Seifert surfaces for $K_1$ and $K_2$, respectively.

We know that $D_1$ and $D_2$ induce a circular locally thin decomposition $D$ on $E(K_1\sharp K_2)$ with circular width given by:

\begin{center}
$cw(E(K_1 \sharp K_2), D)= cw(E(K_1))\sharp cw(E(K_2))= \{ c(G_i \sharp R_1)\} \cup \{c(F_1 \sharp S_j) \}$
\end{center}
Modulo non-increasing order.

Moreover we know that 

$cw(E(K_1 \sharp K_2)) \leq cw(E(K_1))\sharp cw(E(K_2))$.

In order to prove that the equality holds we need to show that $cw(E(K_1))\sharp cw(E(K_2)) \leq cw(E(K_1\sharp K_2))$.

Suppose that $D'$ is a circular decomposition for $E(K_1 \sharp K_2)$ which realizes $cw(E(K_1 \sharp K_2))$. Let $\{T_i\}_{i=1}^m$ be the collection of thin levels and $\{U_i\}_{i=1}^m$ be the collection of thick levels for $D'$. Then $cw(E(K_1\sharp K_2))=\{c(U_i) \}_{i=1}^m$ modulo non-increasing order.

By Proposition \ref{prop1} each $T_i$ and $U_i$ is homeomorphic to a boundary connected sum of Seifert surfaces. Using Corollary \ref{coro1} we can find $s\in \{1,2, ... , m \}$ such that:

\begin{flushleft}
$T_1 \simeq T_1' \sharp T_1''$\\
$U_i \simeq T_1' \sharp U_i''  \qquad 1\leq i < s$\\
$T_i \simeq T_1' \sharp T_i''  \qquad 1\leq i < s$\\
$T_s \simeq T_1 \sharp T_1''$\\
$U_i \simeq U'_i \sharp T_1'' \qquad s\leq i \leq m$\\
$T_i \simeq T_i' \sharp T_1'' \qquad  s< i < m$\\
$T_m=T_1$.
\end{flushleft}

Thus $E(K_1)$ inherits a circular decomposition $D_1'$ with thin levels $\{ T_1' \} \cup \{ T_i'  ; s < i < m\}$ and thick levels $\{ U_i' :  s \leq i \leq m\}$. Then $cw(E(K_1), D_1')=\{ c(U'_i)\}$ modulo non-increasing order.

Also $E(K_2)$ inherits a circular decomposition $D_2'$ with thin levels $\{ T_1'' \} \cup \{ T_i''  ; 1 < i < s\}$ and thick levels $\{ U_i'' :  1 \leq i <s \}$. Then $cw(E(K_2), D_2')=\{ c(U''_i)\}$ modulo non-increasing order.

Also we know;

\begin{equation}
cw(E(K_1)) \leq cw(E(K_1), D_1')    \qquad  cw(E(K_2)) \leq cw(E(K_2), D_2')
\label{1}
\end{equation}

The following equations are true as well

\begin{equation}
c(U_i')+c(T_1'')=c(U_i) \qquad c(U_j'')+c(T_1')=c(U_j)
\label{2}
\end{equation}

\begin{equation}
c(F_1)\leq c(T_1') \qquad c(R_1)\leq c(T_1'')
\label{3}
\end{equation}

Remember that $F_1$ is a minimal genus Seifert surface for $K_1$  and $R_1$ is a minimal genus Seifert surface for $K_2$.

Equations \eqref{1}, \eqref{2} and  \eqref{3} imply the following:

\begin{equation}
\{c(G_i)+c(R_1) \} \leq \{ c(G_i) + c(T_1'')\}\leq \{c(U_i')+c(T_1'') \}=\{ c(U_i)\}
\label{4}
\end{equation}

and 

\begin{equation}
\{c(S_j)+c(F_1) \} \leq \{ c(S_i) + c(T_1')\}\leq \{c(U_j'')+c(T_1') \}=\{ c(U_j)\}
\label{5}
\end{equation}

Modulo non-increasing order.

Notice that ;

\begin{displaymath}
\{c(G_i)+c(R_1) \} \cup \{c(S_j)+c(F_1) \} = cw(E(K_1)) \sharp cw(E(K_2)) 
\end{displaymath}

and

\begin{displaymath}
\{c(U_i) \} \cup \{c(U_j)\} = cw(E(K_1\sharp K_2)). 
\end{displaymath}

Then applying Proposition \ref{ntuples} to the lefthand side and righthand side of equations \eqref{4} and \eqref{5} , we obtain 

\begin{displaymath}
cw(E(K_1)) \sharp cw(E(K_2)) \leq cw(E(K_1\sharp K_2))
\end{displaymath}
 
 This proves the theorem.
\end{proof}

The following question remains open;

\begin{ques}
Does a knot in circular thin position contain a minimal genus Seifert as a thin surface?
\end{ques}

There is evidence that a minimal genus Seifert surface appears in a circular thin position.
All non fibered knots up to ten crossings are almost fibered and the thin surface appearing in the circular thin decomposition is of minimal genus.  

If the answer to the question is affirmative, then the additivity of circular width would be true in general.

\end{document}